\documentclass{amsart}
\usepackage[utf8]{inputenc}

\usepackage{amsmath, amsfonts, amssymb, commath, amsthm, mathrsfs}
\usepackage{hyperref,cleveref}

\theoremstyle{remark}
\newtheorem{rmk}{Remark}[section]
\theoremstyle{definition}
\newtheorem{defn}[rmk]{Definition}
\theoremstyle{plain}
\newtheorem{thm}[rmk]{Theorem}
\newtheorem{prop}[rmk]{Proposition}
\newtheorem{lemma}[rmk]{Lemma}

\renewcommand{\dif}{\mathop{}\!\mathrm{d}}
\newcommand{\real}{\mathbb{R}}
\newcommand{\restr}[2]{\left. #1 \right|_{#2}}
\newcommand{\parens}[1]{\left(#1\right)}
\newcommand{\brackets}[1]{\left[#1\right]}
\newcommand{\angles}[1]{\left\langle #1\right\rangle}
\newcommand{\tr}[2]{\mathrm{tr}_{#1}#2}
\newcommand{\im}[1]{\mathrm{im} \ #1}

\title[Generic nondegeneracy for solutions of the Allen-Cahn equation]{Generic nondegeneracy for solutions of the Allen-Cahn equation under a volume constraint in closed manifolds}
\author{Gustavo de Paula Ramos}
\address{Instituto de Matemática e Estatística, Universidade de São Paulo, Rua do Matão, 1010, 05508-090, São Paulo, SP, Brazil.}
\email{gustavopramos@gmail.com}
\urladdr{http://www.ime.usp.br/~gpramos}

\begin{document}

\begin{abstract}
	Let $M^n$ be a connected closed smooth manifold with $n \geq 2$. We adapt the techniques in \cite{micheletti_pistoia:2009} and \cite{ghimenti_micheletti:2011} to prove the generic nondegeneracy for solutions of the Van der Waals-Allen-Cahn-Hilliard equation under a volume constraint in $M$.
	
	
\end{abstract}

\date{\today}
\maketitle

\section{Introduction and main result}

	Let $(M^n,g)$ be a connected closed smooth Riemannian manifold, where $n \geq 2$. Let $W\colon \real \to \real$ be a function of class $C^2$. Fix $\nu,\epsilon>0$. A pair $(u,\lambda) \in H_g(M) \times \real$ is a solution for the Van der Waals-Allen-Cahn Hilliard equation under volume constraint $\nu$ when
	\begin{equation}\tag{$P_{W,\nu,\epsilon,g}$}\label{eqn:nonlinear-problem}
		\begin{cases}
			-\epsilon^2\Delta_g u + W'(u)=\lambda \\
			\int_M u \dif \mu_g = \nu
		\end{cases},
	\end{equation}
	where $\mu_g$ is a measure induced by $g$ defined on the Borel subsets of $M$ and $H_g(M)$ is a convenient Sobolev space of functions defined in section 2.

	In \cite{benci2020lusternikschnirelman}, the authors establish lower bounds on the number of solutions for $\eqref{eqn:nonlinear-problem}$ in function of  topological invariants of $M$ for sufficiently small $\nu,\epsilon>0$ and under specific hypotheses on the potential function $W$. In particular: if $\eqref{eqn:nonlinear-problem}$ only admits nondegenerate solutions, then Morse theory may be applied to prove that it admits at least $P_M(1)$ solutions, where $P_M(t)$ is the Poincaré polynomial of $M$.
	
	Our main result is that under suitable growth conditions for $W'$ and $W''$, this is indeed the case  generically with respect to $(\epsilon,g) \in ]0,\infty[\times\mathcal{M}^k$, where $1 \leq k < \infty$ and $\mathcal{M}^k$ is the space of Riemannian metrics of class $C^k$ on $M$:
	\begin{thm}\label{thm:main}
	Fix $g_0 \in \mathcal{M}^k$. Suppose that (\ref{eqn:growth-condition-1}) and (\ref{eqn:growth-condition-2}) hold. Then
	\begin{multline*}
		\mathcal{D}^*_{W,\nu}
		=
		\left\{
			(\epsilon,g) \in ]0,\infty[\times\mathcal{M}^k: \text{ any solution } (u,\lambda) \in H_{g_0}(M) \times \real
		\right.
		\\
		\left.
			\text{ for } \eqref{eqn:nonlinear-problem} \text{ is nondegenerate}
		\right\}
	\end{multline*}
	is an open dense subset of $]0,\infty[\times\mathcal{M}^k$.
	\end{thm}
	
	This result is obtained by the application of an abstract transversality theorem through an appropriate adaptation of the techniques in \cite{micheletti_pistoia:2009} and \cite{ghimenti_micheletti:2011} to the context of this article.
	
	More precisely, we say that a solution $(u,\lambda) \in H_g(M) \times \real$ for $\eqref{eqn:nonlinear-problem}$ is \emph{nondegenerate} when the only pair $(v,\Lambda) \in H_g(M) \times \real$ which solves the linearized problem
	\begin{equation}\tag{$Q_{W,\epsilon,g,u}$}\label{eqn:linearized-problem}
		\begin{cases}
			-\epsilon^2\Delta_g v+W''(u)v=\Lambda \\
			\int_M v \dif\mu_g=0
		\end{cases}
	\end{equation}
	is the trivial one $(v,\Lambda)=(0,0)$.
	
	In fact, this notion coincides with the Morse theoretic notion of a nondegenerate critical point for the functional $J_{W,\epsilon,g}\colon H_g(M) \times \real \to \real$ given by
	\[
		J_{W,\epsilon,g}(u,\lambda)
		=
		\int_M \frac{\epsilon^2}{2} g(\nabla u, \nabla u) + W(u) - \lambda u \dif \mu_g - \lambda \nu.
	\]
	Indeed, $J_{W,\epsilon,g}$ is a functional of class $C^2$ for which $(v,\Lambda)$ is a solution for $\eqref{eqn:linearized-problem}$ if, and only if, $\int_M v \dif \mu_g=0$ and $(v,\Lambda) \in \ker \mathrm{Hess} (J_{W,\epsilon,g})_{(u,\lambda)}$. Therefore, $(u,\lambda)$ is a nondegenerate solution for $\eqref{eqn:nonlinear-problem}$ precisely when $(u,\lambda)$ is a nondegenerate critical point of $J_{W,\epsilon,g}$ such that $\int_M u \dif\mu_g=\nu$.
	
	For Differential Geometry, interest for the Van der Waals-Allen-Cahn-Hilliard equation under a volume constraint is justified by the results of \cite{pacard:2003}, where Pacard and Ritoré showed that one can approach constant mean curvature hypersurfaces by the nodal sets of critical points for $J_{W,\epsilon,g,\lambda}$ as $\epsilon \to 0^+$. If we consider critical points without the volume constraint, these sets approach a minimal hypersurface.
	
\subsection*{Acknowledgement}
	
	The author thanks Paolo Piccione for suggesting the topic and discussing drafts of this article.
	
\section{Preliminaries} 
\subsection*{Basic constructions}

	Fix $1 \leq k < \infty$. Denote by $\mathcal{S}^k$ the Banach space of symmetric 2-covectors on $M$ of class $C^k$. The space $\mathcal{M}^k$ of Riemannian metrics on $M$ of class $C^k$ is an open convex cone in $\mathcal{S}^k$.

	Consider any $(\epsilon,g) \in ]0,\infty[\times\mathcal{M}^k$. $(\epsilon,g)$ induces the following inner products on $C^\infty(M)$:
	\[
		\angles{u,v}_g
		=
		\int_M
			g\parens{\nabla u, \nabla v} + uv
		\dif \mu_g;
	\]
	\[
		E_{\epsilon,g}(u,v)
		:=
		\int_M
			\epsilon^2 g\parens{\nabla u, \nabla v} + uv
		\dif \mu_g.
	\]
	
	$H_g(M)$, $H_{\epsilon,g}(M)$ are, respectively, the Hilbert spaces endowed with $\angles{\cdot,\cdot}_g$, $E_{\epsilon,g}$ obtained as completions of $C^\infty(M)$. Similarly: given $1\leq q < \infty$, $L^q_g(M)$ is the Banach space obtained as completion of $C^\infty(M)$ with respect to the norm
	\[
		\norm{u}_{q,g}
		:=
		\parens{
		\int_M
			\abs{u}^q
		\dif \mu_g
		}^{1/q}.
	\]
	
	One may check that the norms induced by $\angles{\cdot,\cdot}_g$, $E_{\epsilon,g}$ on $C^\infty(M)$ are equivalent. In particular, this implies $H_g(M)=H_{\epsilon,g}(M)$ as sets and that the canonical inclusion $H_g(M) \to H_{\epsilon,g}(M)$ is an isomorphism of Banach spaces. The same holds for the canonical inclusion $H_{g'}(M) \to H_g(M)$ for any $g' \in \mathcal{M}^k$. For details, we refer the reader to \cite[Proposition 2.2]{hebey:2000}.
	
\subsection*{Considered setting}

	Suppose that
	\begin{equation}\label{eqn:growth-condition-1}
		\exists K_1 > 0 \ \forall t \in \real, \ \abs{W'(t)}\leq K_1(1+\abs{t}^{p-1});
	\end{equation}
	\begin{equation}\label{eqn:growth-condition-2}
		\exists K_2 > 0 \ \forall t \in \real, \ \abs{W''(t)}\leq K_2(1+\abs{t}^{p-2});
	\end{equation}
	for a certain $p \in \left]2,p_n\right[$, where $p_n=\infty$ for $n=2$, $p_n=(2n)/(n-2)$ for $n\geq 3$.
	
	Fix $g_0 \in \mathcal{M}^k$. Consider any $(\epsilon,g) \in ]0,\infty[\times\mathcal{M}^k$. Due to the Kondrakov theorem, the canonical inclusion $i_{\epsilon,g}\colon H_{\epsilon,g}(M) \to L_g^p(M)$ is a compact operator. Set $p':=p/(p-1)$. We define $A_{\epsilon,g}$ as the adjoint of $i_{\epsilon,g}$ while considering the canonical Banach space isomorphisms $(L^p_g(M))' \simeq L^{p'}_g(M)$ and $H_{\epsilon,g}(M) \simeq (H_{\epsilon,g}(M))'$:
\begin{defn}
	$
		A_{\epsilon,g}=i_{\epsilon,g}^*\colon L_g^{p'}(M) \to H_{\epsilon,g}(M).
	$
\end{defn}
\begin{rmk}\label{rmk:property-of-map-A}
	$A_{\epsilon,g}$ is a compact self-adjoint operator and $E_{\epsilon,g}\parens{A_{\epsilon,g}u,v}=\int_M uv \dif \mu_g$ for any $u,v \in H_g(M)$.
\end{rmk}
	
	For details on lemmas \ref{lemma:map-E} and \ref{lemma:map-A}, we refer the reader to\cite[Lemmas~2.1,~2.3]{micheletti_pistoia:2009}.

\begin{lemma}\label{lemma:map-E}
	\sloppy
	$E\colon ]0,\infty[\times\mathcal{M}^k \to \mathrm{Bil}\parens{H_{g_0}(M)}$ is a map of class $C^1$, where $E(\epsilon,g):=E_{\epsilon,g}$. In particular,
	\begin{multline*}
	\dif E_{(\epsilon,g)}[\eta,h](u,v)
	=
	2\epsilon\eta \int_M g\parens{\nabla u, \nabla v} \dif \mu_g
	+
	\epsilon^2 \int_M b_{g,h}\parens{\nabla u, \nabla v} \dif \mu_g
	+
	\\
	+
	\frac{1}{2}
	\int_M \parens{\tr{g}{h}} uv \dif \mu_g,
	\end{multline*}
	where $b_{g,h}$ is a symmetric 2-covector on $M$ of class $C^k$ given locally by
	\[
		\parens{b_{g,h}}_{ij}=\parens{\tr{g}{h}} g^{ij}/2 - g^{iq}h_{ql}g^{lj}.
	\]
\end{lemma}

\begin{lemma}\label{lemma:map-A}
	$A\colon ]0,\infty[\times\mathcal{M}^k \to B\parens{L_{g_0}^{p'}(M),H_{g_0}(M)}$ is a map of class $C^1$, where $A(\epsilon,g):=A_{\epsilon,g}$. In particular,
	\[
		\dif E_{(\epsilon,g)}[\eta,h]\parens{A_{\epsilon,g}u,v}
		+
		E_{\epsilon,g}\parens{\dif A_{(\epsilon,g)}[\eta,h]u,v}
		=
		\frac{1}{2}
		\int_M \parens{\tr{g}{h}} uv \dif \mu_g.
	\]
\end{lemma}

	\sloppy
	$W'\colon \real \to \real$ is a function of class $C^1$ with suitable growth conditions, so $H_{g_0}(M) \ni u \mapsto W'(u) \in L_{g_0}^{p'}(M)$ is a Nemytskii operator of class $C^1$. For details on this argument, we recommend the reference \cite{kavian:1993}. This implies:
\begin{lemma}
	The Nemytskii operator $B_W\colon H_{g_0}(M)\times\real \to L_{g_0}^{p'}(M)$ given by $B_W(u,\lambda)=\lambda+u-W'(u)$ is a map of class $C^1$. In particular,
	\[
		\dif \parens{B_W}_{(u,\lambda)}[v,\Lambda]=\Lambda+v-vW''(u).
	\]
\end{lemma}

	In the next definition, we identify the space of constant real-valued functions on $M$ with $\real$:
\begin{defn}\label{defn:map-F}
	Let $F_W\colon ]0,\infty[\times \mathcal{M}^k \times (H_{g_0}(M)\setminus \real) \times \real \to H_{g_0}(M) \times \real$ be given by
	\[
		F_W\parens{\epsilon,g,u,\lambda}=\parens{u - A_{\epsilon,g}\circ B_W(u,\lambda),\int_M u \dif \mu_g}.
	\]
\end{defn}

	Using remark \ref{rmk:property-of-map-A}, we can prove that the set of solutions $(u,\lambda) \in (H_{g_0}(M)\setminus \real)\times\real$ for $\eqref{eqn:nonlinear-problem}$ is a level-set of $F_W$:
\begin{rmk}\label{rmk:solution-is-in-F-10}
	 $(u,\lambda) \in H_{g_0}(M)\times\real$ is a solution for $\eqref{eqn:nonlinear-problem}$ if, and only if, $F_W(\epsilon,g,u,\lambda)=(0,\nu)$.
\end{rmk}

\begin{lemma}\label{lemma:map-F}
	$F_W\colon ]0,\infty[\times \mathcal{M}^k \times (H_{g_0}(M)\setminus \real) \times \real \to H_{g_0}(M) \times \real$ is a map of class $C^1$. In particular,
	\begin{multline*}
		\dif \parens{F_W}_{\parens{\epsilon,g,u,\lambda}} [\eta,h,v,\Lambda]
		=
		\\
		=
		\left(
		v
		-
		A_{\epsilon,g} \circ\dif\parens{B_W}_{(u,\lambda)}[v,\Lambda]
		-
		\dif A_{(\epsilon,g)}[\eta,h]\circ B_W(u,\lambda)
		,
		\right.
		\\
		\left.
		,
		\int_M \frac{1}{2}\parens{\tr{g}{h}}u + v \dif \mu_g
		\right)
	\end{multline*}
\end{lemma}

\section{Proof of main result}
	
	Consider the following abstract transversality theorem:
\begin{thm}\label{thm:transversality}\cite[Theorem 5.4]{henry:2005}
	Let $X,Y,Z$ be real Banach spaces and $U, V$ be respective open subsets of $X,Y$. Let $F\colon V \times U \to Z$ be a map of class $C^m$, where $m \geq 1$. Let $z_0 \in \im{F}$. Suppose that
\begin{enumerate}
	\item
		Given $y \in V$, $F(y,\cdot)\colon x \mapsto F(x,y)$ is a Fredholm map of index $l< m$, i.e., $\dif F(y,\cdot)_x\colon X \to Z$ is a Fredholm operator of index $l$ for any $x \in U$;
	\item
		$z_0$ is a regular value of $F$, i.e., $\dif F_{\parens{y_0,x_0}}\colon Y \times X \to Z$ is surjective for any $(y_0,x_0) \in F^{-1}(z_0)$;
	\item
		Let $\iota\colon F^{-1}(z_0) \to Y \times X$ be the canonical embedding and $\pi_Y\colon Y \times X \to Y$ be the projection of the first coordinate. Then $\pi_Y\circ \iota\colon F^{-1}(z_0) \to Y$ is $\sigma$-proper, i.e., $F^{-1}(z_0)=\bigcup_{s=1}^\infty C_s$, where given $s=1,2,...$, $C_s$ is a closed subset of $F^{-1}(z_0)$ and $\restr{\pi_Y\circ\iota}{C_s}$ is proper.
\end{enumerate}	
	Then the set $\{ y \in V: z_0 \text{ is a regular value of } F(y,\cdot) \}$ is an open dense subset of $V$.
\end{thm}

	The first step to prove our main result is the lemma that follows, in which we restrict ourselves to nonconstant solutions:
	\begin{lemma}\label{lemma:main}
	Fix $g_0 \in \mathcal{M}^k$. Suppose that (\ref{eqn:growth-condition-1}) and (\ref{eqn:growth-condition-2}) hold. Then
	\begin{multline*}
		\mathcal{D}_{W,\nu}
		=
		\left\{
			(\epsilon,g) \in ]0,\infty[\times\mathcal{M}^k: \text{ any solution } (u,\lambda) \in (H_{g_0}(M)\setminus\real) \times \real
		\right.
		\\	
		\left.
			\text{ for } \eqref{eqn:nonlinear-problem} \text{ is nondegenerate}
		\right\}
	\end{multline*}
	is an open dense subset of $]0,\infty[\times\mathcal{M}^k$.
	\end{lemma}

	Its proof consists of a direct application of the abstract transversality theorem. Specifically, we consider $X=Z=H_{g_0}(M)\times \real$, $Y=V=]0,\infty[\times\mathcal{S}^k$, $U=(H_{g_0}(M)\setminus \real) \times \real$, $F=F_W$ and $z_0=(0,\nu)$. We verify that its hypotheses hold in section 4.

	After analysing the constant solutions for $\eqref{eqn:nonlinear-problem}$, we refine lemma \ref{lemma:main} to prove our main result:
\begin{proof}[Proof of Theorem \ref{thm:main}]
	Let $(\epsilon,g) \in ]0,\infty[\times\mathcal{M}^k$ and $\mathcal{U}$ be a neighborhood of $(\epsilon,g)$ in $]0,\infty[\times\mathcal{M}^k$.
	
	$\mathcal{D}^*_{W,\nu}\cap\mathcal{U}$ is not empty. Indeed, let $(\overline{\epsilon},\overline{g}) \in \mathcal{D}_{W,\nu}\cap\mathcal{U}$. If $\parens{P_{\nu,\overline{\epsilon},\overline{g},W}}$ does not admit constant solutions, then $(\overline{\epsilon},\overline{g}) \in \mathcal{D}^*_{W,\nu}\cap\mathcal{U}$. Otherwise, the volume constraint shows that the unique constant solution is $\nu/\mu_{\overline{g}}(M)$. This is a degenerate solution if, and only if, $\eqref{eqn:linearized-problem}$ admits a nontrivial solution. This only happens when there exists $j=1,2,...$ such that
	\begin{equation}\label{eqn:cst-is-deg}
		\overline{\epsilon}^2
		=
		-\frac{W''(\nu/\mu_{\overline{g}}(M))}{\alpha_j(\overline{g})},
	\end{equation}
	where $\mathcal{E}_{\overline{g}}=\set{\alpha_j\parens{\overline{g}}: j=1,2,...}$ is the set of nonzero eigenvalues of $-\Delta_{\overline{g}}$. $\mathcal{E}_{\overline{g}}$ is a discrete subset of $]0,\infty[$, so there exists $\hat{\epsilon}>0$ such that $(\hat{\epsilon},\overline{g}) \in \mathcal{D}_{W,\nu}\cap\mathcal{U}$ and (\ref{eqn:cst-is-deg}) does not hold for any positive integer $j$. This implies $(\hat{\epsilon},\overline{g}) \in \mathcal{D}^*_{W,\nu}$.
	
	$\mathcal{D}^*_{W,\nu}$ is an open subset of $]0,\infty[\times\mathcal{M}^k$. Indeed, let $(\hat{\epsilon},\hat{g}) \in \mathcal{D}^*_{W,\nu}$. If $(P_{W,\nu,\hat{\epsilon},\hat{g}})$ does not admit constant solutions, the result is a corollary of lemma \ref{lemma:main}. Otherwise, note that $\mathcal{M}^k \ni g \mapsto W''(\nu/\mu_g(M)) \in \real$ and $\mathcal{M}^k \ni g \mapsto \alpha_j(g) \in \real$ are continuous maps for any positive integer $j$, so $(\hat{\epsilon},\hat{g})$ admits a neighborhood $\mathcal{V}$ in $]0,\infty[\times \mathcal{M}^k$ in which the constant solutions are nondegenerate. To conclude, $\mathcal{V} \cap \mathcal{D}_{W,\nu}$ is a neighborhood of $(\epsilon,g)$ in $]0,\infty[\times\mathcal{M}^k$ for which the respective Allen-Cahn equation does not admit degenerate solutions.
\end{proof}

\section{Technical steps}

	For a pair $(\epsilon,g)\in]0,\infty[\times\mathcal{M}^k$, let $F_{W,\epsilon,g}\colon(H_{g_0}(M)\setminus \real) \times \real \to H_{g_0}\times \real$ be given by
	$
		F_{W,\epsilon,g}(u,\lambda)
		=
		F_W(\epsilon,g,u,\lambda).
	$
	We adopt similar notation when fixing other variables.

	In lemma \ref{lemma:first-hypothesis}, we shall verify that the first hypothesis of theorem \ref{thm:transversality} holds. With that objective in mind, consider the following preliminary result:
\begin{lemma}\label{lemma:first-hypothesis-prelim}
	\sloppy
	Let $g \in \mathcal{M}^k$, $C_g\colon H_{g_0}(M) \to \real$ be given by $C_g(v)=\int_M v \dif \mu_g$ and $T_g\colon H_{g_0}(M)\times\real \to H_{g_0}(M)\times\real$ be given by $T_g(v,\Lambda)=\parens{v,C_g(v)}$. Then $T_g$ is a Fredholm operator of index $0$.
\end{lemma}
\begin{proof}
	$C_g$ is a linear functional, so $\mathrm{codim} \ \ker C_g =1$ in $H_{g_0}(M)$. This implies
	\[
		\mathrm{codim} \ T_g(\ker C_g  \times \real)=2
	\]
	in $H_{g_0}(M) \times \real$.
	
	\[
		T_g(\ker C_g  \times \real)\cap T_g(\real(1,0))=0,
	\]
	so
	\[
		\mathrm{codim} \ [T_g(\ker C_g  \times \real) + T_g(\real(1,0))] = \mathrm{codim} \ T_g(\ker C_g  \times \real) - 1=1.
	\]
	
	\[
		H_{g_0}(M) \times \real = (\ker C_g  \times \real) \oplus (\real(1,0)),
	\]
	so $\mathrm{codim} \ \im T_g=1.$ $\ker T_g=\set{0} \times \real$, so $T_g$ is a Fredholm operator of index $0$.
\end{proof}

\begin{lemma}\label{lemma:first-hypothesis}
	Given $(\epsilon,g) \in ]0,\infty[\times\mathcal{M}^k$, $F_{W,\epsilon,g}$ is a Fredholm map of index $0$.
\end{lemma}
\begin{proof}
	Fix $(u,\lambda) \in H_{g_0}(M) \times \real$ and let $K_{W,\epsilon,g,u,\lambda}\colon H_{g_0}(M)\times\real \to H_{g_0}(M)\times\real$ be given by
	\[
		K_{W,\epsilon,g,u,\lambda}(v,\Lambda)
		=
		\parens{
			A_{\epsilon,g}\circ \dif \parens{B_W}_{(u,\lambda)}[v,\Lambda]
			,
			0
		}.
	\]
	
	$\dif\parens{F_{W,\epsilon,g}}_{(u,\lambda)}=T_g-K_{W,\epsilon,g,u,\lambda}$, where $T_g$ was defined in lemma \ref{lemma:first-hypothesis-prelim}. Therefore, it suffices to prove that $K_{W,\epsilon,g,u,\lambda}$ is a compact operator to conclude that $\dif\parens{F_{W,\epsilon,g}}_{(u,\lambda)}$ is a Fredholm operator with index $0$. This is indeed the case, because $A_{\epsilon,g}$ is a compact operator.
\end{proof}

	Let us examine the second hypothesis of the abstract transversality theorem. Let $(\epsilon,g,u,\lambda)\in F_W^{-1}(0,\nu)$. To conclude that $\dif\parens{F_W}_{(\epsilon,g,u,\lambda)}$ is surjective, it suffices to show that 
	\begin{equation}\label{eqn:inclusion}
		\brackets{\im{\dif\parens{F_{W,\epsilon,g}}_{(u,\lambda)}}}^\perp
		\subset
		\im{\dif\parens{F_{W,\epsilon,u,\lambda}}_g},
	\end{equation}
	which we shall prove in lemma \ref{lemma:nondeg}.
	
	The following defines an inner product on $H_{g_0}(M)\times\real$:
	\[
		\angles{\parens{u_1,t_1},\parens{u_2,t_2}}_{\epsilon,g}'
		=
		E_{\epsilon,g}\parens{u_1,u_2}+t_1t_2.
	\]
	This allows us to establish the characterization:
\begin{rmk}\label{rmk:kernel-differential}
	Let $(\epsilon,g) \in ]0,\infty[\times\mathcal{M}^k$ and $(u,\lambda), \ (v,\Lambda)\in H_{g_0}(M)\times\real$. Then
	\[
		(v,\Lambda) \in \brackets{\im{\dif\parens{F_{W,\epsilon,g}}_{(u,\lambda)}}}^\perp
		\text{ if, and only if, }
		(v,-\Lambda)
		\text{ is a solution for }
		\eqref{eqn:linearized-problem}.
	\]
\end{rmk}
	
	We use this characterization to prove inclusion (\ref{eqn:inclusion}):
\begin{lemma}\label{lemma:nondeg}
	Let $(\epsilon,g,u,\lambda) \in F_W^{-1}(0,\nu)$. Let $(v,-\Lambda) \in H_{g_0}(M)\times\real$ be a solution for $\eqref{eqn:linearized-problem}$. If
	\[
		\angles{
			\dif \parens{F_{W,\epsilon,u,\lambda}}_g[h]
			,
			(v,\Lambda)
		}_{\epsilon,g}'
		=
		0
	\]
	for all $h \in \mathcal{S}^k$, then $(v,\Lambda)=(0,0)$.
\end{lemma}
\begin{proof}
	Due to lemmas \ref{lemma:map-E}, \ref{lemma:map-A} and \ref{lemma:map-F}, the equation on the statement is rewritten
\begin{equation}\label{eqn:conta-lema-deg}
	\int_M \epsilon^2 b_{g,h}\parens{\nabla u, \nabla v}
	+
	\frac{\parens{\tr{g}{h}}}{2}
	\brackets{
	\parens{W'(u)-\lambda}v
	+
	\Lambda u
	}
	\dif \mu_g
	=
	0,
\end{equation}
	where we recall that $b_{g,h}$ is a symmetric 2-covector on $M$ of class $C^k$ given locally by
	\[
		\parens{b_{g,h}}_{ij}=\parens{\tr{g}{h}} g^{ij}/2 - g^{iq}h_{ql}g^{lj}.
	\]
	
	An argument with normal coordinates centered at arbitrary $x \in M$ which considers specific perturbations of $g$ proves that
\begin{equation}\label{eqn:conta-lema-deg-2}
	g\parens{\nabla u, \nabla v}=b_{g,h}\parens{\nabla u, \nabla v} = 0 \in L^1_{g_0}(M)
\end{equation}
for any $h \in \mathcal{S}^k$. For details, see \cite[Lemma~12]{ghimenti_micheletti:2011}.

	Taking $h=\varphi g$ for arbitrary $\varphi \in C^\infty(M)$ shows that (\ref{eqn:conta-lema-deg}) and (\ref{eqn:conta-lema-deg-2}) imply
	\begin{equation}\label{eqn:conta-lema-deg-3}
		\parens{W'(u)-\lambda}v+\Lambda u = 0 \in L^1_{g_0}(M).
	\end{equation}
	
	On one hand: integrating the equation above (\ref{eqn:conta-lema-deg-3}) yields
	\begin{equation}\label{eqn:conta-lema-deg-4}
		\int_M \parens{W'(u)-\lambda}v \dif \mu_g=-\Lambda\nu.
	\end{equation}
	
	On the other hand: taking into account (\ref{eqn:conta-lema-deg-2}) and the fact that $u$ is a weak solution for $-\epsilon^2\Delta_g u + W'(u)=\lambda$,
	\[
		\int_M W'(u)v \dif \mu_g=\int_M \lambda v \dif \mu_g.
	\]
	
	\sloppy
	$\nu \neq 0$, so the last equation and (\ref{eqn:conta-lema-deg-4}) imply $\Lambda=0$. Due to (\ref{eqn:conta-lema-deg-3}), $\Lambda=0$ implies 		\[	
	\parens{W'(u)-\lambda}v = 0 \in L^1_{g_0}(M).
	\]
	
	If $\lambda-W'(u)\equiv 0$, then $u$ is a weak solution for $-\epsilon^2\Delta_gu=0$ -- which only happens with a constant $u$. We do not consider constant solutions, so $\lambda-W'(u)$ does not vanish identically. Due to proposition \ref{prop:regularity}, $u,v$ are functions of class $C^1$. Therefore, $\lambda-W'(u)$ is a continuous function which does not vanish identically.
	
	In particular, $v$ vanishes in a nonempty open subset of $M$. In this context, we can use strong unique continuation (\cite[Theorem~A.5]{pigola_rigoli_setti:2008}) in problem $\eqref{eqn:linearized-problem}$ to conclude that $v = 0 \in H_{g_0}(M)$.
\end{proof}

\begin{prop}\label{prop:regularity}
	Fix $(\epsilon,g)\in ]0,\infty[\times\mathcal{M}^k$ and $\alpha \in ]0,1[$. If $(u,\lambda) \in H_{g_0}(M)\times\real$ is a solution for $\eqref{eqn:nonlinear-problem}$, then $u \in C^{1,\alpha}(M)$. If it also holds that $(v,\Lambda) \in H_{g_0}(M)\times\real$ is a solution for $\eqref{eqn:linearized-problem}$, then $v \in C^{1,\alpha}(M)$.
\end{prop}
\begin{proof}
	Regularity is a local problem, so we fix a coordinate system $\rho\colon\Omega\to\real^n$ where the $g^{ij}$s are bounded and $\Omega$ is the coordinate open subset of $M$. Let $\tilde{u}\colon\rho(\Omega)\to\real$ be the local expression of $u$. Note that $\tilde{u}$ is a weak solution for
	\[
		-\epsilon^2 \partial_i\parens{g^{ij}\partial_j \tilde{u}}+\epsilon^2b^i\parens{\partial_i \tilde{u}}
		+
		W'\parens{\tilde{u}}
		=
		\lambda,
	\]
	where $b^i=\partial_j\parens{g^{ij}}+g^{ij}\Gamma_{kj}^k$ for any $i=1,...,n$.
	
	A slight adaptation of \cite[Lemma~B.3]{struwe:2010} shows that $\tilde{u} \in L^q\parens{\rho(\Omega)}$ for every $q<\infty$. Arguing as in \cite[Theorem~12.2.2]{jost:2013}, one may show that given $q>1$, $W'\parens{\tilde{u}} \in L^q\parens{\rho(\Omega)}$ implies $u \in H^{2,q}\parens{\rho(\Omega)}$. To conclude, we use the Sobolev Embedding Theorem.
\end{proof}

	The third hypothesis is proved analogously as\cite[Lemma~11]{ghimenti_micheletti:2011}:
\begin{lemma}
	$\pi_Y\circ \iota\colon F_W^{-1}(0,\nu) \to Y$ is $\sigma$-proper, where $\pi_Y$ and $\iota$ are defined in theorem \ref{thm:transversality}.
\end{lemma}
\begin{proof}
	Given $s=1,2,...$; let
	\[
		C_s=\parens{[1/s,s] \times \overline{\mathscr{B}_{s}} \times \overline{I(0,s)\setminus B(\real,1/s)}\times [-s,s]}\cap F_W^{-1}(0,\nu),
	\]
	where $\mathscr{B}_{s}, I(0,s)$ are respective open balls in $\mathcal{S}^k, H_{g_0}(M)$ centered at $0$ with radius $s$ and
	\[
		B(\real,1/s)
		=
		\set{
			u \in H_{g_0}(M):
			\inf_{v \in \real} \norm{u+v}_{H_{g_0}} < 1/s
		}.
	\]
	
	Fix a positive integer $s$. Let us prove that $\restr{\pi_Y \circ \iota}{C_s}$ is a proper map. Let $\set{\parens{\epsilon_n,g_n,u_n,\lambda_n}}_n \subset C_s$ be a sequence such that $\lim_n g_n = g \in \mathcal{M}^k$, $\lim_n \epsilon_n=\epsilon \in [1/s,s]$ and given $n$, $\parens{u_n,\lambda_n}$ is a solution for $(P_{W,\nu,\epsilon_n,g_n})$.
	
	We claim that $\parens{\parens{u_n,\lambda_n}}_n$ has a convergent subsequence. Due to the Kondrakov theorem, the canonical inclusion $i_{\epsilon,g,t}\colon H_{\epsilon,g_0}(M)\to L_{g_0}^t(M)$ is a compact operator for any $t \in \left]2,p_n\right[$, so $\parens{u_n}_n$ converges in $L_{g_0}^t(M)$ up to subsequence to a certain $u \in L_{g_0}^t(M)$. $\parens{\lambda_n}_n$ is bounded, so it converges up to a subsequence to a certain $\lambda \in [-s,s]$. Arguing as in lemma \ref{lemma:first-hypothesis}, we see that $\lim_n A_{\epsilon,g}\circ B_W\parens{u_n,\lambda_n} = A_{\epsilon,g}\circ B_W\parens{u,\lambda}$. We can use the Mean Value Inequality and lemma \ref{lemma:map-A} to prove that, in fact, $\lim_n A_{\epsilon_n,g_n}\circ B_W\parens{u_n,\lambda_n} = A_{\epsilon,g}\circ B_W\parens{u,\lambda}$.
\end{proof}

\bibliographystyle{alpha}
\bibliography{bibliography}
\end{document}